\documentclass[11pt,reqno]{amsart}

\usepackage[a4paper,margin=3cm]{geometry}
\usepackage[dvipsnames]{xcolor}
\usepackage[english]{babel}
\usepackage[babel]{microtype}

\usepackage{amsmath}
\usepackage{amssymb}
\usepackage{amsthm}
\usepackage[bb=boondox]{mathalpha}

\theoremstyle{plain}
\newtheorem{theorem}{Theorem}[section]
\newtheorem{proposition}[theorem]{Proposition}
\newtheorem*{proposition*}{Proposition}
\newtheorem*{theorem*}{Theorem}
\newtheorem{lemma}[theorem]{Lemma}
\newtheorem*{lemma*}{Lemma}
\newtheorem{corollary}[theorem]{Corollary}
\newtheorem{openproblem}[theorem]{Open Problem}

\theoremstyle{definition}

\newtheorem*{remark*}{Remark}

\usepackage{mathrsfs}
\usepackage{stmaryrd}
\usepackage{tikz}
\usetikzlibrary{arrows.meta,decorations.pathreplacing}
\colorlet{cladecol}{MidnightBlue}
\colorlet{hicol}{BrickRed}
\colorlet{subsetcol}{ForestGreen!70!black}
\usepackage{enumerate}
\usepackage{mathtools}
\usepackage[full]{textcomp}
\mathtoolsset{showonlyrefs}

\numberwithin{equation}{section}

\usepackage[T1]{fontenc}
\usepackage{libertinus}
\usepackage[utf8]{inputenc}

\usepackage[symbol]{footmisc}
\usepackage{comment}
\usepackage{appendix}
\usepackage{graphicx}
\usepackage{float}
\usepackage{needspace}

\usepackage[
  breaklinks=true,
  pagebackref=true,
]{hyperref}
\hypersetup{
  colorlinks=true,
  pdfpagemode=UseNone,
  citecolor=ForestGreen,
  linkcolor=MidnightBlue,
  urlcolor=RoyalBlue,
  pdfstartview=FitH,
  pdftitle={The Height of Discrete-Time Critical Beta-Splitting Trees},
  pdfauthor={Heng Ma}
}

\usepackage{etoolbox}
\patchcmd{\paragraph}
  {\normalfont}
  {{\normalfont\fontseries{sb}\selectfont}}
  {}
  {\PackageError{paragraph-style}
    {Failed to patch paragraph}
    {Check your document class and heading packages.}}
\makeatletter
\patchcmd{\paragraph}
  {\z@\z@}
  {\z@{2pt}}
  {}
  {\PackageError{paragraph-style}
    {Failed to patch paragraph spacing}
    {Check your document class and heading packages.}}
\makeatother

\def\N{\mathbb{N}}

\def\E{\mathbf{E}}
\renewcommand{\P}{\mathbf P}

\newcommand{\ind}[1]{\mathbf{1}_{\{#1\}}}

\newcommand*{\dif}{\ensuremath{\mathop{}\!\mathrm{d}}}

\newcommand{\I}{\mathcal I}
\newcommand{\DTCS}{\operatorname{DTCS}}
\newcommand{\CTCS}{\operatorname{CTCS}}
\newcommand{\Exp}{\operatorname{Exp}}

\definecolor{RoadmapRevisionColor}{RGB}{175,40,40}

\title[Height of critical beta-splitting trees]
{The Height of Discrete-Time Critical Beta-Splitting Trees}
\author{Heng Ma}
\address[Heng Ma]
{Faculty of Data and Decision Sciences, Technion - Israel's Institute of
Technology, Haifa, 32000, Israel.}
\email{hengmamath(at)gmail(dot)com}
\urladdr{\url{https://hengmamath.github.io}}
\date{\today}

\begin{document}

\begin{abstract}
We determine the  asymptotic height of the discrete-time
critical beta-splitting tree. Let $L_n^*$ denotes the height of the tree with $n$ leaves, defined as   the maximum graph distance from the root to a leaf. Then,  
\[
  \frac{L_n^*}{(\log n)^2}
  \longrightarrow
  C_{\mathrm{ht}}:=\min_{\theta>1}
  \frac{\theta}{2\{\psi(\theta)+\gamma\}} \approx  0.976 
\]
almost surely and in $L^{p}$ for every fixed $p>1$ as $n \to \infty$. Here
  $\psi$ is the digamma function and $\gamma$ is Euler's constant. This answers \cite[Open Problem~4]{AldousJansonII} of Aldous and Janson.
\end{abstract}

\maketitle

\section{Introduction}
\label{sec:introduction}

The critical beta-splitting tree belongs to the family of cladogram models
introduced by Aldous~\cite{AldousCladograms}. We determine the first-order
constant for its maximum root-to-leaf graph distance.

Write
$[n]=\{1,\ldots,n\}$ and $h_m=\sum_{j=1}^m j^{-1}$, with $h_0=0$.
Starting from the initial clade $[n]$,  
each clade of size $m\geq2$ splits into
an ordered pair of children, with left-child size distribution
\begin{equation}\label{eq:splitting-rule}
  q(m,i)=\frac{1}{2h_{m-1}} \frac{m}{i(m-i)}
  =\frac{1}{2h_{m-1}}
   \Bigl(  \frac1i+\frac1{m-i}\Bigr)  ,
  \qquad 1\leq i<m.
\end{equation}
Given size $i$, the left child receives a uniform $i$-element subset of the
parent's labels and the right child receives the complement. The children
evolve independently by the same rule; singletons become leaves.
Denote this tree by $\DTCS(n)$; see Figure~\ref{fig:dtcs-definition}.
Let $L_n(r)$
denote the height of leaf $\{r\}$, defined as its graph distance from the root $[n]$.  The height of $\DTCS(n)$ is given by
\begin{equation}\label{eq:Ln-star-definition}
  L_n^*:=\max_{1\leq r\leq n}L_n(r).
\end{equation} 
By the consistency property of Aldous and Janson
\cite[Theorem~2.3 and Remark~2.7]{AldousJansonIII},
the trees $(\DTCS(n))_{n\geq1}$ can be coupled so that deleting
leaf $n+1$ from $\DTCS(n+1)$ and suppressing its parent
recovers $\DTCS(n)$; if the parent is the root, its remaining
child becomes the new root.
Throughout, we use this coupling, whose fragmentation
realization is described in Section~\ref{sec:fragmentation}.

Aldous and Pittel~\cite[Theorem~1.5]{AldousPittel} proved a high-probability
upper bound  $L_{n}^{*} \le C (\log n)^2$ with $C \approx 42.9$.  
Aldous and Janson~\cite[Open Problem~4]{AldousJansonII}
asked whether $L_n^*/(\log n)^2$ converges in probability and for the
value of its limit. We answer this question as follows.

Let $\psi=\Gamma'/\Gamma$ be the digamma function and
$\gamma :=\lim\limits_{m\to\infty}(h_m-\log m)$ be Euler's constant. Set
\begin{equation}\label{eq:kappa-definition}
  \kappa(\theta):=\psi(\theta)+\gamma ,
  \qquad \theta>1.
\end{equation}

\begin{theorem}\label{thm:main} 
For every fixed $1\leq p<\infty$, as $n \to \infty$,
\begin{equation}\label{eq:main-limit}
  \frac{L_n^*}{(\log n)^2}\xrightarrow[n \to \infty]{\mathrm{a.s.},\, L^p }C_{\mathrm{ht}} :=  \min_{\theta>1}\frac{\theta}{2\kappa(\theta)} = \frac{1}{2 \psi'(\theta_*)}. 
\end{equation} 
Here $\theta_*$ is the unique solution of the equation $\theta \psi'(\theta) = \kappa(\theta)$ on $(1,\infty)$.
\end{theorem} 

Numerically, 
$C_{\mathrm{ht}}\approx0.976$.
For comparison, the height of a uniformly chosen leaf, divided by
$(\log n)^2$, converges in probability to $3/\pi^2\approx0.304$
\cite{AldousPittel}.
Thus the maximum and the typical leaf height have the same order of growth,
but their leading constants differ by a factor of more than three.
In their discussion of Open Problem~4,
Aldous and Janson~\cite[Section~3.13]{AldousJansonII}
raised the possibility, based on numerical observations for a greedy path,
that the leading coefficient might equal $3/\pi^2$.
Theorem~\ref{thm:main} shows that the tree height has a strictly larger
first-order constant.

\begin{figure}[t]
\centering
\begin{tikzpicture}[
  font=\footnotesize,
  br/.style={line width=.65pt,color=cladecol},
  hbr/.style={line width=1.15pt,color=hicol},
  sbr/.style={line width=1.15pt,color=subsetcol},
  cl/.style={rectangle,fill=cladecol,inner sep=1.3pt},
  lf/.style={circle,draw=cladecol,fill=white,line width=.5pt,inner sep=1.1pt},
  dt/.style={circle,fill=black!70,inner sep=.9pt},
  sdt/.style={circle,fill=subsetcol,inner sep=1.15pt},
  cbox/.style={line width=.7pt,rounded corners=2pt,
               dash pattern=on 2.2pt off 1.8pt},
  lab/.style={inner sep=1.2pt},
  clab/.style={inner sep=1pt,font=\scriptsize,text=cladecol},
]

\begin{scope}[x=0.62cm,y=0.62cm]
  \node[lab] at (3,6.62) {parent clade $\mathsf C=[7]$};
  \foreach \x/\n in {0/1,1/2,2/3,3/4,4/5,5/6,6/7}{
     \node[dt] at (\x,5.65) {};
     \node[lab,anchor=south,font=\tiny,yshift=1pt] at (\x,5.65) {\n};}
  \node[sdt] at (2,5.65) {};
  \node[sdt] at (4,5.65) {};
  \draw[cbox,color=cladecol] (-0.55,5.32) rectangle (6.55,6.24);

  \draw[sbr] (2.55,5.22) -- (1.1,4.48);
  \draw[br]  (3.45,5.22) -- (4.75,4.48);

  \node[sdt] at (0.75,3.9) {}; \node[lab,anchor=south,font=\tiny,yshift=1pt] at (0.75,3.9) {3};
  \node[sdt] at (1.45,3.9) {}; \node[lab,anchor=south,font=\tiny,yshift=1pt] at (1.45,3.9) {5};
  \draw[cbox,color=subsetcol] (0.35,3.57) rectangle (1.85,4.49);
  \node[lab,anchor=north,font=\scriptsize,text=subsetcol,yshift=-2pt] at (1.1,3.57)
     {$i=2$};

  \foreach \x/\n in {3.85/1,4.35/2,4.85/4,5.35/6,5.85/7}{
     \node[dt] at (\x,3.9) {}; \node[lab,anchor=south,font=\tiny,yshift=1pt] at (\x,3.9) {\n};}
  \draw[cbox,color=cladecol] (3.5,3.57) rectangle (6.2,4.49);
  \node[lab,anchor=north,font=\scriptsize,yshift=-2pt] at (4.85,3.57)
     { $m{-}i=5$};

  \node[lab,anchor=north,align=center,font=\scriptsize,text=black!65] at (3,2.2)
     {given $i=2$, the left child is a uniform $2$-element\\
      subset of $[7]$; the right child gets the rest};

  \node[anchor=north,font=\small] at (3,0.25) {(a) one split};
\end{scope}

\begin{scope}[xshift=6.2cm,x=0.78cm,y=0.85cm]
  \draw[line width=.5pt,-{Straight Barb[length=3.6pt,width=3pt]}]
        (-1.05,-0.2) -- (-1.05,4.5);
  \foreach \g in {0,1,2,3,4}{
     \draw[line width=.5pt] (-1.13,\g) -- (-0.97,\g);
     \node[lab,anchor=east,font=\scriptsize] at (-1.17,\g) {\g};}
  \node[lab,anchor=south,font=\scriptsize] at (-1.05,4.56) {height};

  \draw[br] (2.875,0) -- (2.875,1);
  \draw[br] (1.5,1) -- (2.875,1);
  \draw[br] (2.875,1) -- (4.25,1);
  \draw[br] (1.5,1) -- (1.5,2);
  \draw[br] (4.25,1) -- (4.25,2);
  \draw[br] (1,2) -- (1.5,2);
  \draw[br] (1.5,2) -- (2,2);
  \draw[br] (3,2) -- (4.25,2);
  \draw[br] (4.25,2) -- (5.5,2);
  \draw[br] (5.5,2) -- (5.5,3);
  \draw[br] (4.5,3) -- (5.5,3);
  \draw[br] (5.5,3) -- (6.5,3);
  \draw[br] (4.5,3) -- (4.5,4);
  \draw[br] (6.5,3) -- (6.5,4);
  \draw[br] (4,4) -- (4.5,4);
  \draw[br] (4.5,4) -- (5,4);
  \draw[br] (6,4) -- (6.5,4);
  \draw[br] (6.5,4) -- (7,4);

  \draw[hbr] (2.875,0) -- (2.875,1);
  \draw[hbr] (2.875,1) -- (4.25,1);
  \draw[hbr] (4.25,1) -- (4.25,2);
  \draw[hbr] (4.25,2) -- (5.5,2);
  \draw[hbr] (5.5,2) -- (5.5,3);
  \draw[hbr] (5.5,3) -- (6.5,3);
  \draw[hbr] (6.5,3) -- (6.5,4);
  \draw[hbr] (6.5,4) -- (7,4);

  \foreach \p in {(2.875,0),(1.5,1),(4.25,1),(5.5,2),(4.5,3),(6.5,3)}
     {\node[cl] at \p {};}

  \node[lf] at (1,2) {}; \node[lab,anchor=south,font=\scriptsize,yshift=1.5pt] at (1,2) {$5$};
  \node[lf] at (2,2) {}; \node[lab,anchor=south,font=\scriptsize,yshift=1.5pt] at (2,2) {$3$};
  \node[lf] at (3,2) {}; \node[lab,anchor=south,font=\scriptsize,yshift=1.5pt] at (3,2) {$2$};
  \node[lf] at (4,4) {}; \node[lab,anchor=south,font=\scriptsize,yshift=1.5pt] at (4,4) {$7$};
  \node[lf] at (5,4) {}; \node[lab,anchor=south,font=\scriptsize,yshift=1.5pt] at (5,4) {$4$};
  \node[lf] at (6,4) {}; \node[lab,anchor=south,font=\scriptsize,yshift=1.5pt] at (6,4) {$6$};
  \node[lf,draw=hicol,line width=.8pt] at (7,4) {}; \node[lab,anchor=south,font=\scriptsize,yshift=1.5pt,text=hicol] at (7,4) {$1$};

  \node[clab,anchor=east,xshift=-2pt] at (2.875,-0.05) {$[7]$};
  \node[clab,anchor=east,xshift=-2pt,text=subsetcol] at (1.5,1) {$\{3,5\}$};
  \node[clab,anchor=south west,xshift=2pt,yshift=1pt,text=hicol] at (4.25,1) {$\{1,2,4,6,7\}$};
  \node[clab,anchor=south west,xshift=1pt,text=hicol] at (5.5,2) {$\{1,4,6,7\}$};

  \node[lab,anchor=west,font=\scriptsize,text=hicol] at (7.1,4) {$L_7(1)=4=L_7^*$};

  \node[anchor=north,font=\small] at (3,-0.75) {(b) a realization of $\DTCS(7)$};
\end{scope}
\end{tikzpicture}
\caption{One splitting step (a) and the tree it generates (b): the split
in (a) is the root split of (b), every edge has unit length, and the
highlighted ancestral path attains the height $L_7^*=4$.}
\label{fig:dtcs-definition}
\end{figure}
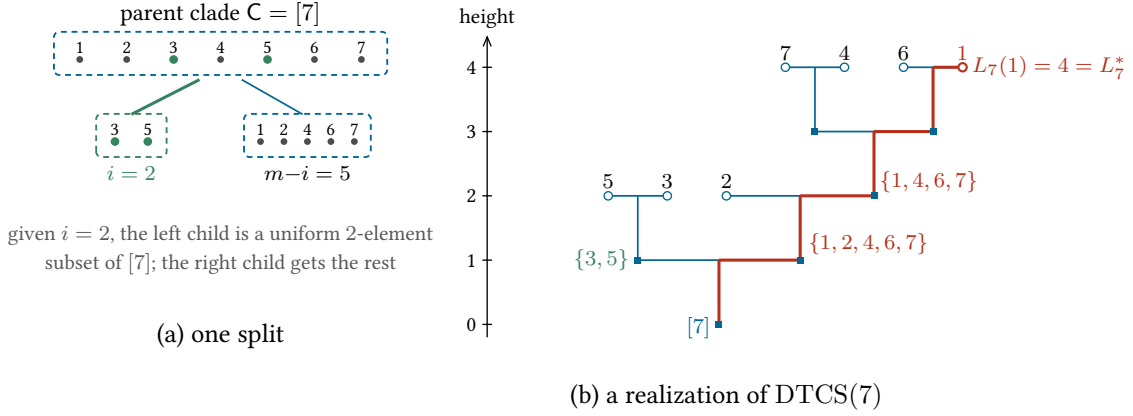

\subsection{Related works}
\label{sec:related-work}

Embedding discrete random trees into continuous-time branching processes
is a powerful approach to their analysis, dating back at least to
Pittel~\cite{Pittel1994}.
In our setting, the continuous-time
critical beta-splitting tree $\CTCS(n)$
\cite{AldousPittel, AldousJansonII, AldousJansonIII}
is obtained by assigning a lifetime to each internal clade of $\DTCS(n)$,
that is, each clade containing at least two labels.
Conditional on $\DTCS(n)$, these lifetimes are independent exponentials,
with rate $h_{m-1}$ for a clade of size $m\geq2$.
This choice of holding times allows one to draw on the
theory of homogeneous fragmentation processes, since $\CTCS(n)$ can be
realized as a finite restriction of such a process
(see \cite[Section 4]{AldousJansonIII} or Section~\ref{sec:fragmentation} for details).

The height and time-height of a uniformly chosen leaf have both been
studied. Let $D_n(r)$ denote the time-height of
leaf $\{r\}$ in $\CTCS(n)$, namely the sum of the lifetimes of its
ancestor clades, excluding the leaf itself.
Aldous and Pittel~\cite{AldousPittel}
established moment asymptotics and central limit theorems for both
heights. In particular, let $U\sim\mathrm{Unif}(0,1)$ be independent of
the tree, so that $\lceil nU\rceil$ is a uniformly chosen label. Then 
  \begin{align}
    \frac{L_n(\lceil nU\rceil)-\frac{3}{\pi^2}(\log n)^2}
         {(\log n)^{3/2}}
    &\xrightarrow[n\to\infty]{\mathrm{law}}
      \mathcal N\! \Bigl(  0,\frac{144\zeta(3)}{\pi^6}\Bigr) ,\label{uleaf-clt-1}\\
    \frac{D_n(\lceil nU\rceil)-\frac{6}{\pi^2}\log n}
         {(\log n)^{1/2}}
    &\xrightarrow[n\to\infty]{\mathrm{law}}
      \mathcal N\! \Bigl(  0,\frac{432\zeta(3)}{\pi^6}\Bigr)  ,
  \end{align} 
where $\zeta$ is the Riemann zeta function.
Using Mellin analysis, Aldous and
Janson~\cite{AldousJansonIV} obtained full asymptotic
expansions for both mean heights, together with sharper variance
asymptotics and large deviation estimates for the time-height.
Kolesnik~\cite{KolesnikContraction} gave alternative
proofs of both central limit theorems via the contraction method,
including quantitative bounds on the normal approximation.

The ancestral clade sizes of a uniformly chosen leaf form the harmonic
descent chain. Iksanov~\cite{IksanovHarmonicDescent} connected this chain
with regenerative composition structures, giving a 
proof of its limiting hitting probabilities using renewal theory. Building on this
representation, Iksanov, Nikitin, and
Yakymiv~\cite{IksanovNikitinYakymiv} proved a joint central
limit theorem for the height and time-height of a uniformly chosen
leaf, together with the numbers of ancestral clade-size decrements of
each fixed size.

Let $D_n^*:=\max_{r\in[n]}D_n(r)$ be the maximum time-height.
Applying Joseph's height theorem~\cite[Theorem~1]{Joseph} to the
integer-time skeleton of the homogeneous fragmentation representation
of Aldous and Janson~\cite[Section~4]{AldousJansonIII} yields
$D_n^*/\log n\to2$ in probability. Thus $D_n^*$ has very different asymptotic behavior comparing with $L_n^*$, because  $L_n^*$ grows on the
scale $(\log n)^2$ by \cite{AldousPittel}. 
The convergence in distribution of $D_n^*-2\log n$ to a randomly
shifted Gumbel law was established independently by Chorbadzhiyska, Minchev, and
Savov~\cite[Theorem~2.12]{ChorbadzhiyskaMinchevSavov} and in our companion
paper~\cite{MaCTCS}. 
The companion paper also describes a freezing transition between $r=1$
and $r\geq2$ in the centering and limiting laws of the $r$-shattering
time---the first time at which all clades have size at most $r$---and
establishes limit theorems for the associated extremal processes and
clade-count dynamics.

\subsection{Roadmap of the proof} 
\label{sec:roadmap}   
Write $\ell_{\mathsf C}$ for the lifetime of an internal clade
$\mathsf C$ in  $\CTCS(n)$ introduced in
Section~\ref{sec:related-work}.
For $r\in[n]$, let $\mathcal A_n(r)$ be the set of the ancestor clades of leaf
$\{r\}$,  excluding the leaf itself. 
Using this notation, the height of $\{r\}$ in $\DTCS(n)$ and  the time-height of  $\{r\}$ in $\CTCS(n)$   are respectively
\[
  L_n(r)=\sum_{\mathsf{C}\in\mathcal A_n(r)}1,
  \qquad
  D_{n}(r)=\sum_{\mathsf{C}\in\mathcal A_n(r)}\ell_{\mathsf{C}}.
\]

\paragraph{Rescaling the clade lifetimes.}
To compare leaf height with these lifetimes, we multiply each lifetime of an internal clade 
by its splitting rate. This choice is inspired by Dyszewski, Johnston, Palau,
and Prochno~\cite[Section~3.3]{DyszewskiJohnstonPalauProchno}, who construct
a Crump--Mode--Jagers branching process from a self-similar fragmentation
via time changes depending on each fragment’s mass. Set  
\[ E_{\mathsf{C}}:=h_{|\mathsf{C}|-1}\ell_{\mathsf{C}} \quad \text{ for each clade  $\mathsf{C}$ with $|\mathsf{C}|\ge 2$.}  \] 
Then conditionally on $\DTCS(n)$, $\{E_{\mathsf{C}}: |\mathsf{C}| \ge 2\}$ are i.i.d. standard exponentials; see Figure~\ref{fig:dtcs-rescaling}.   Define
\begin{equation}\label{original:eq:Q-leaf-definition}
  \widetilde L_n(r):=\sum_{\mathsf{C}\in\mathcal A_n(r)} E_{\mathsf{C}}  \quad , \quad 
  \widetilde L_n^*:=\max_{1\leq r\leq n}\widetilde L_n(r).
\end{equation} 
Lemma~\ref{lem:marked-graph-comparison} shows that, 
\[
  \max_{r\leq n}|L_n(r)-\widetilde L_n(r)|
  =O((\log n)^{3/2})
  \qquad\text{almost surely}.
\]
This reduces the almost-sure height asymptotic to
proving that $\widetilde L_n^*/(\log n)^2\to C_{\mathrm{ht}}$ almost surely.

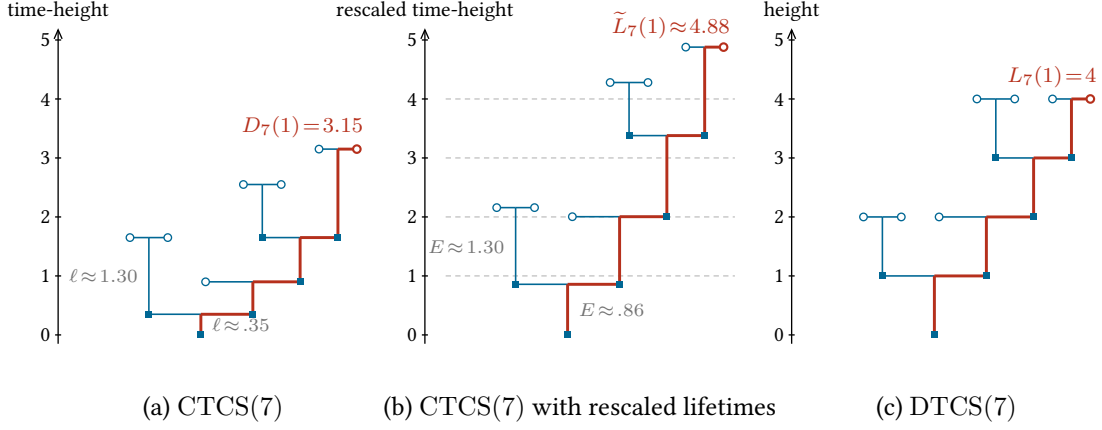
\begin{figure}[t]
\centering
\begin{tikzpicture}[
  font=\footnotesize,
  br/.style={line width=.6pt,color=cladecol},
  hbr/.style={line width=1.1pt,color=hicol},
  cl/.style={rectangle,fill=cladecol,inner sep=1.15pt},
  lf/.style={circle,draw=cladecol,fill=white,line width=.5pt,inner sep=1pt},
  guide/.style={line width=.4pt,color=black!30,dash pattern=on 2pt off 1.6pt},
  lab/.style={inner sep=1pt},
]

\begin{scope}[x=0.5cm,y=0.78cm]
  \draw[line width=.5pt,-{Straight Barb[length=3.4pt,width=3pt]}]
        (-0.9,-0.15) -- (-0.9,5.2);
  \foreach \t in {0,1,2,3,4,5}{
     \draw[line width=.5pt] (-0.98,\t) -- (-0.82,\t);
     \node[lab,anchor=east,font=\scriptsize] at (-1.02,\t) {\t};}
  \node[lab,anchor=south,font=\scriptsize] at (-0.9,5.26) {time-height};

  \draw[br] (2.875,0) -- (2.875,0.35);
  \draw[br] (1.5,0.35) -- (2.875,0.35);
  \draw[br] (2.875,0.35) -- (4.25,0.35);
  \draw[br] (1.5,0.35) -- (1.5,1.65);
  \draw[br] (4.25,0.35) -- (4.25,0.90);
  \draw[br] (1,1.65) -- (1.5,1.65);
  \draw[br] (1.5,1.65) -- (2,1.65);
  \draw[br] (3,0.90) -- (4.25,0.90);
  \draw[br] (4.25,0.90) -- (5.5,0.90);
  \draw[br] (5.5,0.90) -- (5.5,1.65);
  \draw[br] (4.5,1.65) -- (5.5,1.65);
  \draw[br] (5.5,1.65) -- (6.5,1.65);
  \draw[br] (4.5,1.65) -- (4.5,2.55);
  \draw[br] (6.5,1.65) -- (6.5,3.15);
  \draw[br] (4,2.55) -- (4.5,2.55);
  \draw[br] (4.5,2.55) -- (5,2.55);
  \draw[br] (6,3.15) -- (6.5,3.15);
  \draw[br] (6.5,3.15) -- (7,3.15);

  \draw[hbr] (2.875,0) -- (2.875,0.35);
  \draw[hbr] (2.875,0.35) -- (4.25,0.35);
  \draw[hbr] (4.25,0.35) -- (4.25,0.90);
  \draw[hbr] (4.25,0.90) -- (5.5,0.90);
  \draw[hbr] (5.5,0.90) -- (5.5,1.65);
  \draw[hbr] (5.5,1.65) -- (6.5,1.65);
  \draw[hbr] (6.5,1.65) -- (6.5,3.15);
  \draw[hbr] (6.5,3.15) -- (7,3.15);

  \foreach \p in {(2.875,0),(1.5,0.35),(4.25,0.35),(5.5,0.90),(4.5,1.65),(6.5,1.65)}
     {\node[cl] at \p {};}
  \foreach \p in {(1,1.65),(2,1.65),(3,0.90),(4,2.55),(5,2.55),(6,3.15)}
     {\node[lf] at \p {};}
  \node[lf,draw=hicol,line width=.8pt] at (7,3.15) {};

  \node[lab,font=\tiny,text=black!55,anchor=west,xshift=3pt] at (2.875,0.17) {$\ell\!\approx\!.35$};
  \node[lab,font=\tiny,text=black!55,anchor=east,xshift=-3pt] at (1.5,1.0) {$\ell\!\approx\!1.30$};
  \node[lab,font=\scriptsize,text=hicol,anchor=south east] at (7.25,3.30) {$D_7(1)\!=\!3.15$};

  \node[anchor=north,font=\small] at (3.2,-0.85) {(a) $\CTCS(7)$};
\end{scope}

\begin{scope}[xshift=4.85cm,x=0.5cm,y=0.78cm]
  \foreach \g in {1,2,3,4}{ \draw[guide] (-0.35,\g) -- (7.3,\g); }
  \draw[line width=.5pt,-{Straight Barb[length=3.4pt,width=3pt]}]
        (-0.9,-0.15) -- (-0.9,5.2);
  \foreach \g in {0,1,2,3,4,5}{
     \draw[line width=.5pt] (-0.98,\g) -- (-0.82,\g);
     \node[lab,anchor=east,font=\scriptsize] at (-1.02,\g) {\g};}
  \node[lab,anchor=south,font=\scriptsize] at (-0.9,5.26) {rescaled time-height};

  \draw[br] (2.875,0) -- (2.875,0.858);
  \draw[br] (1.5,0.858) -- (2.875,0.858);
  \draw[br] (2.875,0.858) -- (4.25,0.858);
  \draw[br] (1.5,0.858) -- (1.5,2.158);
  \draw[br] (4.25,0.858) -- (4.25,2.004);
  \draw[br] (1,2.158) -- (1.5,2.158);
  \draw[br] (1.5,2.158) -- (2,2.158);
  \draw[br] (3,2.004) -- (4.25,2.004);
  \draw[br] (4.25,2.004) -- (5.5,2.004);
  \draw[br] (5.5,2.004) -- (5.5,3.379);
  \draw[br] (4.5,3.379) -- (5.5,3.379);
  \draw[br] (5.5,3.379) -- (6.5,3.379);
  \draw[br] (4.5,3.379) -- (4.5,4.279);
  \draw[br] (6.5,3.379) -- (6.5,4.879);
  \draw[br] (4,4.279) -- (4.5,4.279);
  \draw[br] (4.5,4.279) -- (5,4.279);
  \draw[br] (6,4.879) -- (6.5,4.879);
  \draw[br] (6.5,4.879) -- (7,4.879);

  \draw[hbr] (2.875,0) -- (2.875,0.858);
  \draw[hbr] (2.875,0.858) -- (4.25,0.858);
  \draw[hbr] (4.25,0.858) -- (4.25,2.004);
  \draw[hbr] (4.25,2.004) -- (5.5,2.004);
  \draw[hbr] (5.5,2.004) -- (5.5,3.379);
  \draw[hbr] (5.5,3.379) -- (6.5,3.379);
  \draw[hbr] (6.5,3.379) -- (6.5,4.879);
  \draw[hbr] (6.5,4.879) -- (7,4.879);

  \foreach \p in {(2.875,0),(1.5,0.858),(4.25,0.858),(5.5,2.004),(4.5,3.379),(6.5,3.379)}
     {\node[cl] at \p {};}
  \foreach \p in {(1,2.158),(2,2.158),(3,2.004),(4,4.279),(5,4.279),(6,4.879)}
     {\node[lf] at \p {};}
  \node[lf,draw=hicol,line width=.8pt] at (7,4.879) {};

  \node[lab,font=\tiny,text=black!55,anchor=west,xshift=3pt] at (2.875,0.43) {$E\!\approx\!.86$};
  \node[lab,font=\tiny,text=black!55,anchor=east,xshift=-3pt] at (1.5,1.5) {$E\!\approx\!1.30$};
  \node[lab,font=\scriptsize,text=hicol,anchor=south east] at (7.25,4.98) {$\widetilde L_7(1)\!\approx\!4.88$};

  \node[anchor=north,font=\small,align=center] at (3.2,-0.85)
     {(b) $\CTCS(7)$ with rescaled lifetimes};
\end{scope}

\begin{scope}[xshift=9.7cm,x=0.5cm,y=0.78cm]
  \draw[line width=.5pt,-{Straight Barb[length=3.4pt,width=3pt]}]
        (-0.9,-0.15) -- (-0.9,5.2);
  \foreach \g in {0,1,2,3,4,5}{
     \draw[line width=.5pt] (-0.98,\g) -- (-0.82,\g);
     \node[lab,anchor=east,font=\scriptsize] at (-1.02,\g) {\g};}
  \node[lab,anchor=south,font=\scriptsize] at (-0.9,5.26) {height};

  \draw[br] (2.875,0) -- (2.875,1);
  \draw[br] (1.5,1) -- (2.875,1);
  \draw[br] (2.875,1) -- (4.25,1);
  \draw[br] (1.5,1) -- (1.5,2);
  \draw[br] (4.25,1) -- (4.25,2);
  \draw[br] (1,2) -- (1.5,2);
  \draw[br] (1.5,2) -- (2,2);
  \draw[br] (3,2) -- (4.25,2);
  \draw[br] (4.25,2) -- (5.5,2);
  \draw[br] (5.5,2) -- (5.5,3);
  \draw[br] (4.5,3) -- (5.5,3);
  \draw[br] (5.5,3) -- (6.5,3);
  \draw[br] (4.5,3) -- (4.5,4);
  \draw[br] (6.5,3) -- (6.5,4);
  \draw[br] (4,4) -- (4.5,4);
  \draw[br] (4.5,4) -- (5,4);
  \draw[br] (6,4) -- (6.5,4);
  \draw[br] (6.5,4) -- (7,4);

  \draw[hbr] (2.875,0) -- (2.875,1);
  \draw[hbr] (2.875,1) -- (4.25,1);
  \draw[hbr] (4.25,1) -- (4.25,2);
  \draw[hbr] (4.25,2) -- (5.5,2);
  \draw[hbr] (5.5,2) -- (5.5,3);
  \draw[hbr] (5.5,3) -- (6.5,3);
  \draw[hbr] (6.5,3) -- (6.5,4);
  \draw[hbr] (6.5,4) -- (7,4);

  \foreach \p in {(2.875,0),(1.5,1),(4.25,1),(5.5,2),(4.5,3),(6.5,3)}
     {\node[cl] at \p {};}
  \foreach \p in {(1,2),(2,2),(3,2),(4,4),(5,4),(6,4)}
     {\node[lf] at \p {};}
  \node[lf,draw=hicol,line width=.8pt] at (7,4) {};

  \node[lab,anchor=south east,font=\scriptsize,text=hicol] at (7.25,4.14) {$L_7(1)\!=\!4$};

  \node[anchor=north,font=\small] at (3.2,-0.85) {(c) $\DTCS(7)$};
\end{scope}
\end{tikzpicture}
\caption{The tree of Figure~\ref{fig:dtcs-definition}(b) drawn with three
edge lengths: the $\CTCS(7)$ lifetimes $\ell_{\mathsf C}$ (a), the
rescaled lifetimes $E_{\mathsf C}=h_{|\mathsf C|-1}\ell_{\mathsf C}$ (b),
and unit edges (c). The $E_{\mathsf C}$ are i.i.d.\ standard
exponentials, so (b) differs from (c) only by random edges of mean one in
place of length one; the dashed lines mark the integers.}
\label{fig:dtcs-rescaling}
\end{figure}

\paragraph{From leaf height to a time integral.}
Following Aldous and Janson~\cite[Section~4]{AldousJansonIII},
we realize all $\CTCS(n)$ using a single homogeneous interval fragmentation
$\I=(\mathcal I_t)_{t\geq0}$, described in Section~\ref{sec:fragmentation},
and an independent sequence $U_1,U_2,\ldots$ of i.i.d. uniform points in
$(0,1)$. 
Roughly speaking  $(\mathcal{I}_t)_{t\ge0}$ is a stochastic process taking values in open subsets of $(0,1)$ having branching property, with $\I_{0}=(0,1)$ and
$\I_{t+s}\subset  \I_{t}$ for all $s,t\geq0$.   
 Labels in $[n]$ that lie in the same interval component of $\I_t$ form the clade at time $t$. The resulting clade genealogy, equipped with the clades' birth and death
times, has the law of $\CTCS(n)$, see Lemma \ref{lem:finite-restriction}.

Write 
$\I_t(U_r)$ for the interval  fragment containing $U_r$, $N_n(t,U_r)$ for the
number of labels in $[n]$ that lie in $\I_t(U_r)$. Summing the rescaled lifetimes along the ancestral path of leaf
$\{r\}$ gives 
\begin{equation}\label{eq:rescaled-Lnr}
  \widetilde L_n(r)
  =\sum_{\mathsf C\in\mathcal A_n(r)}E_{\mathsf C}
  =\int_0^{D_n(r)}h_{N_n(t,U_r)-1}\,\dif t.
\end{equation}
Lemma~\ref{lem:marked-graph-comparison} gives the formal identity and its proof.

\paragraph{Proof sketch for the upper bound.} To control the integral \eqref{eq:rescaled-Lnr}, one has to  compare the number of
labels in a component of $\I_t$   with its length.
Lemma~\ref{lem:all-interval-occupancy} provides events
$G_n$ occurring for all sufficiently large $n$ almost surely, on which
\[ \text{every interval $I\subset(0,1)$ contains at most
$C(n|I|+\log n)$ labels in $[n]$. } \]  
 This single event controls all
interval fragments of $\I_t$ at all times $t$, and gives  
\[
  h_{N_n(t,U_r)-1}
  \leq  \bigl(    \log  (n  |\mathcal{I}_t(U_r)| )  \bigr)_+ + O(\log\log n).
\]

Lemma~\ref{lem:isolation-time} gives $D_n^*\leq4\log n$ eventually 
so the accumulated error in \eqref{eq:rescaled-Lnr} is
$o_{\mathrm{a.s.}}((\log n)^2)$.\footnote{
For a deterministic sequence $b_n>0$, we write
$R_n=o_{\mathrm{a.s.}}(b_n)$ if $R_n/b_n\to0$ almost surely.
}
Let $\I_t^{(1)}$ be a largest interval component of $\I_t$ and set
$M(t):=\log(1/|\I_t^{(1)}|)$. Since
$|\I_t(U_r)|\leq|\I_t^{(1)}|=e^{-M(t)}$, we get
\[
  \widetilde L_n^*
  \leq \int_0^\infty(\log n-M(t))_+\,\dif t
       + o_{\mathrm{a.s.}}((\log n)^2) .
\]
The SLLN of largest fragment \cite[Equation~(9)]{BertoinAsymptotic} (see also \cite{KyprianouLaneMorters}) gives
$M(t)/t\to v_*:=\kappa'(\theta_*)$ almost surely. Consequently, we obtain 
\[
  \begin{aligned}
    \int_0^\infty(\log n-M(t))_+\,\dif t
    &= [1+o_{\mathrm{a.s.}}(1)]
       \int_0^{(\log n)/v_*}(\log n-v_*t)\,\dif t\\
    &=\frac{1+o_{\mathrm{a.s.}}(1)}{2v_*}(\log n)^2,
  \end{aligned}
\]
which yields the desired upper bound.

\paragraph{Proof sketch for the lower bound.}
The largest fragments at different times need not belong to a single
ancestral lineage. Consequently, the upper-bound integral need not
be attained by any one leaf, and the largest fragment SLLN alone
does not give the lower bound.
We address this by choosing a nested sequence of interval fragments
on a fixed time mesh, following the bounded-lookahead idea of
McDiarmid~\cite[Algorithm~A3]{McDiarmid1990} and
Addario-Berry and Maillard~\cite[Section~2]{AddarioBerryMaillard};
see also Shi~\cite[Theorem~1.3]{ShiBRW} for a similar idea.

To bound $\widetilde L_n^*$ from below via
\eqref{eq:rescaled-Lnr}, we again compare label counts with interval
fragment lengths.
Lemma~\ref{lem:all-interval-occupancy} also gives the
lower estimate on the same event $G_n$:
\begin{equation}\label{original:eq:occupancy-2}
  \#\{j\in[n]:U_j\in I\}\geq cn|I|
  \quad \forall \, \text{  interval }I\subset(0,1)
  \text{ with } |I|\geq C \tfrac{\log n}{n}
\end{equation}
Here $c,C>0$ are absolute constants.

Fix $T>0$, start with $J_0=(0,1)$, and at time $jT$ select a largest
descendant $J_j$ of the previously chosen fragment $J_{j-1}$.  The branching  property (or fragmentation  property) implies the  
subsequent
evolution of $J_{j-1}$, after rescaling it to $(0,1)$, is independent of  its past and has the same law
as $\I$. 
Hence the successive logarithmic mass losses $\{\log ({|J_{j-1}|}/{|J_{j}|})\}_{j \ge 1}$ are independent with
the law of $M(T)$.  The SLLN gives
$ \log ({1}/{|J_{j}|}) /j\to \mu_T := \E[M(T)]$ almost surely.

Let $ \mathcal{K}_n$ be the last index for which
$|J_{\mathcal{K}_n}|\geq(\log n)^2/n$. The previous SLLN  implies
$\mathcal{K}_n= [\frac{1}{\mu_T}+o_{\mathrm{a.s.}}(1)] \log n$ almost surely.
On $G_n$, this terminal interval   $J_{\mathcal{K}_n}$ contains at least
$c(\log n)^2$ labels. Choose one of them. Its ancestor throughout time window
$[(j-1)T, j T]$ contains $J_j$, so the lower
occupancy estimate \eqref{original:eq:occupancy-2} and $h_{m-1}\geq\log m$ give,  
\begin{equation}
     \widetilde L_n^{*} \ge \widetilde L_n(r)
     \geq T\sum_{j=1}^{\mathcal{K}_n} \bigl[ 
       \log n-\log(1/|J_j|)   -\Theta(1) \bigr]. \label{original:eq-intro-tildeL-1}
\end{equation}
Using $\log(1/|J_j|)/j\to\mu_T$ shows that
the right-hand side, divided by $(\log n)^2$, converges a.s.
to $T/(2\mu_T)$. In Lemma~\ref{lem:largest-fragment-speed} we show 
the LLN $M(t)/t\to v_*$ also holds in $L^1$, so
$\mu_T/T\to v_*$ as $T\to\infty$.
Since $G_n$ occurs eventually, the lower
bound $T/(2\mu_T)$ holds a.s. for each fixed $T$.
Taking a countable intersection over positive integers $T$ and then
letting $T\to\infty$ yields the matching lower bound $1/(2v_*)$.

\subsection{Further Questions}
\label{sec:further-questions}

Although we focus on discrete critical beta-splitting trees, we expect
the reduction of tree height to an integrated splitting rate
to apply more broadly. Our companion paper~\cite{MaCTCS} establishes
limit theorems for the time-height of  a wide class of
homogeneous fragmentation trees. A natural next step is to determine
the height asymptotics of the corresponding genealogical trees.

A motivating example is the phase transition at $\beta=-1$ in
Aldous's beta-splitting tree family \cite{AldousCladograms}   with parameter $\beta>-2$. In \cite[Proposition 4]{AldousCladograms}  proved the   height  has logarithmic
scale $\log n$ for $\beta>-1$
, and polynomial scale
$n^{-\beta-1}$ for $-2<\beta<-1$.
At the critical value $\beta=-1$, Aldous and Pittel~\cite{AldousPittel}
established the intermediate scale $(\log n)^2$. 

How does this picture extend to general fragmentation trees?
Under certain regular variation assumptions for the dislocation measure, Haas, Miermont, Pitman, and
Winkel~\cite[Theorem~2]{HaasMiermontPitmanWinkel} already obtain
height scales of the form $n^\alpha L(n)$, with $0<\alpha<1$
and $L$ slowly varying. Beyond these results, it would be interesting
to develop a general classification of height growth and phase
transitions. Which dislocation measures give logarithmic,
squared-logarithmic, or polynomial growth, and what additional
intermediate scales can occur?

Another direction is to study finer asymptotics for $L_n^*$.
The largest-fragment law admits a logarithmic refinement
\cite[Theorem~2.2]{KyprianouLaneMorters}, but our lower-bound construction
provides no quantitative error control as the block length $T$ grows.
The present method therefore leaves second-order corrections and
fluctuations unresolved, leading to the following question.

\begin{openproblem}[Second-order asymptotics]
\label{prob:height-second-order}
Determine the second-order fluctuations of $L_n^*$.
In particular, motivated by \eqref{uleaf-clt-1}, is there a nondegenerate random variable $Z$ such that 
\[
  \frac{L_n^*-C_{\mathrm{ht}}(\log n)^2}{(\log n)^{3/2}}
  \xrightarrow[n\to\infty]{\mathrm{law}} Z ?
\]
If so is the limiting law
Gaussian? 

A competing possibility is a correction of order $(\log n)^{4/3}$.
This is motivated by the results on consistent minimal displacement of
branching random walks due to Fang and
Zeitouni~\cite[Theorem~1]{FangZeitouni} and Faraud, Hu, and
Shi~\cite[Theorem~1.4]{FaraudHuShi}.  In the fragmentation representation, integrating a
deviation of order $(\log n)^{1/3}$ over a time interval of order
$\log n$ suggests a correction of order $(\log n)^{4/3}$.
Does this scale govern the almost-sure second-order behavior of $L_n^*$
under the coupling above?
These are alternative possibilities: an almost-sure
$O((\log n)^{4/3})$ bound for the centered height would force the
displayed normalization to converge to zero.
\end{openproblem}

 \section{Fragmentation representation}
\label{sec:fragmentation}

\subsection{Interval fragmentation}

Aldous and Janson~\cite[Section~4.2]{AldousJansonIII}
identified the continuous-time critical beta-splitting trees as finite
restrictions of a homogeneous partition fragmentation process. We use
an interval fragmentation realization of this representation to compare clade sizes
with fragment masses simultaneously over all times and lineages.

A homogeneous \emph{interval fragmentation} $\I=(\I_{t})_{t\geq0}$ takes
values in the open subsets of $(0,1)$. It starts from $\I_{0}=(0,1)$
and evolves by breaking its interval components
into smaller intervals, so that $\I_{t+s}\subset \I_{t}$ for every $s,t\geq0$.
Conditional on $\I_{t}$, the future evolutions inside its components
are independent copies of the original process, rescaled in space to
their respective intervals. Homogeneity means that their time scales
are unchanged; see
\cite[Definitions~1--2]{BertoinSelfSimilar}
with self-similarity index $\alpha=0$.  We call  
the connected components of $\I_{t}$   the fragments present at
time $t$. The mass of a fragment is its interval length.

Such a process is characterized by its interval
dislocation measure $\nu$ on open subsets of $(0,1)$,
and its left and right erosion coefficients; see
\cite[Theorems~2.7 and~3.5]{Basdevant}. Roughly speaking,
independently for each interval component, dislocation events are driven by
a Poisson point process with intensity
$\dif t\,\nu(\dif U)$: at an atom $(t,U)$, the component
$(a,b)$ is replaced by the open set $a+(b-a)U$, whose connected
components are the offspring intervals.
 
To construct the interval fragmentation corresponding to the critical beta-splitting
trees, we take zero erosion and choose the interval dislocation
measure to be the pushforward of
\begin{equation}\label{eq:oriented-dislocation}
  \rho(\dif x)=\frac{\dif x}{2x(1-x)} \ , \quad 0<x<1,
\end{equation}
under the map $x\mapsto(0,x)\cup(x,1)$.
Thus, a dislocation with parameter $x$ replaces a component $(a,b)$
by the two intervals
$(a,a+x(b-a))$ and $(a+x(b-a),b)$.
The process is conservative: almost surely, for every $t\geq0$,
the component lengths of $\mathcal I_t$ sum to one.

The following lemma combines Bertoin's process-level paintbox
representation \cite[Section~3.2, Lemmas~5--6]{BertoinSelfSimilar}
with the fragmentation representation of $\CTCS$ due to
Aldous and Janson~\cite{AldousJansonIII}.

\begin{lemma} \label{lem:finite-restriction} Let $\I=(\I_t)_{t\geq0}$ be the interval fragmentation described above,
let $U_1,U_2,\ldots$ be i.i.d. uniform points of $(0,1)$, independent
of $\I$.  
For each $n\geq1$, define the clades at time $t\geq0$ as the
nonempty sets
\[
  \{j\in[n]:U_j\in J\},
\]
where $J$ ranges over the interval components of $\I_t$;
any remaining labels form singleton clades.

The genealogy of these clades, determined by set inclusion and
equipped with the lifetimes and the
left-to-right order inherited from the intervals, has the law
of $\CTCS(n)$.
\end{lemma}

\begin{proof}
By \cite[Theorem~3.5 and  Proposition~4.3]{Basdevant},
the partition-valued process on $\N$ induced by $\I$ and
$(U_j)_{j\geq1}$ is a homogeneous fragmentation with zero erosion
and dislocation measure given by the image of $\rho$ under
$x\mapsto(\max\{x,1-x\},\min\{x,1-x\},0,\ldots)$.
This is exactly the measure identified in
\cite[equation~(4.12)]{AldousJansonIII}, so their representation
\cite[Section~4.2]{AldousJansonIII} gives the genealogy and lifetimes
of $\CTCS(n)$ after forgetting the left-to-right order.
Finally, by the symmetry of $\rho$,  
 two child intervals at each split are placed in either order with equal
probability, independently at different splits. This agrees with
the ordering convention in $\CTCS(n)$.
\end{proof}

Throughout the remainder of the paper, we use the coupling
described in Lemma~\ref{lem:finite-restriction}.

\subsection{The integrated splitting rate}
\label{sec:clock}
\label{sec:occupancy}  
For an interval $I\subseteq(0,1)$, let
\begin{equation}\label{eq:interval-occupancy}
  N_n(I):=\#\{j\in[n]:U_j\in I\}
\end{equation}
be the number of sampled labels it contains. For $x\in\I_t$, write
$\I_t(x)$ for the interval fragment containing $x$, and set
\begin{equation}\label{eq:sample-clade-size}
  N_n(t,x):=N_n(\I_t(x)).
\end{equation} 
Note that, almost surely, $N_n(t,U_r)$ is well-defined
simultaneously for all $t \ge 0$ and $r \ge 1$. Indeed, by
conservativity of $\I$, for each pair of integers $k,r \ge 1$,
$  \P(U_r \notin \mathcal{I}_k)
  \le \E [|\mathcal{I}_0 \setminus \mathcal{I}_k| ]
  = 0$. 
Taking a countable intersection over $k,r$, and using
the fact that $(\mathcal{I}_t)_{t \ge 0}$ is decreasing,
yields the claim.

By Lemma~\ref{lem:finite-restriction}, $N_n(t,U_r)$ is the
size of the clade containing label $r$ at time $t$. Hence, the time-height $D_n(r)$  of the
leaf $\{r\}$ introduced in Section \ref{sec:related-work} can be rewritten as 
\begin{equation}\label{eq:isolation-time}
  D_n(r)=\inf\{t\geq0:N_n(t,U_r)=1\} .
\end{equation} 
Recall that $\mathcal A_n(r)$ is the set of ancestor clades of leaf
$\{r\}$, excluding the leaf itself, and that $\ell_{\mathsf C}$ is
the lifetime of a clade $\mathsf C$. The rescaled lifetimes and heights
introduced in Section~\ref{sec:roadmap} are
\begin{equation}\label{eq:Q-leaf-definition}
  E_{\mathsf C}=h_{|\mathsf C|-1}\ell_{\mathsf C},
  \qquad
  \widetilde L_n(r)=\sum_{\mathsf C\in\mathcal A_n(r)}E_{\mathsf C},
  \qquad
  \widetilde L_n^*=\max_{r\in[n]}\widetilde L_n(r).
\end{equation} 

\begin{lemma}\label{lem:marked-graph-comparison}
Almost surely, for every $n\geq2$ and $r\in[n]$,
\begin{equation}\label{eq:height-integral}
  \widetilde L_n(r)
  =\int_0^{D_n(r)}h_{N_n(t,U_r)-1}\,\dif t.
\end{equation}
Moreover, under the coupling fixed above, as $n\to\infty$,
\begin{equation}\label{eq:uniform-marked-graph-comparison}
  \Delta_n:=\max_{r\leq n}|L_n(r)-\widetilde L_n(r)|
  =O((\log n)^{3/2})
  \qquad\text{almost surely}.
\end{equation}
In particular, we have  $  {|L_n^*-\widetilde L_n^*|} / {(\log n)^2} \to   0 $ alsmost surely as $n \to \infty$.
\end{lemma} 

Before proving Lemma \ref{lem:marked-graph-comparison}, we first state a useful lemma controlling $  D_n^* = \max_{r \in [n]}D_n(r)$.

\begin{lemma}\label{lem:isolation-time}
For every $n\geq2$ and $t\geq0$,
\begin{equation}\label{eq:isolation-tail}
  \P(D_n^*>2\log n+t)\leq\frac12e^{-t}.
\end{equation} 
\end{lemma}

\begin{proof}[Proof of Lemma~\ref{lem:marked-graph-comparison}]
For each $r\in[n]$, the ancestral clades partition the time interval
$[0,D_n(r))$ into their lifetimes. On the lifetime interval of
$\mathsf C\in\mathcal A_n(r)$, we have
$N_n(t,U_r)=|\mathsf C|$. Hence
\[
  \int_0^{D_n(r)}h_{N_n(t,U_r)-1}\,\dif t
  =\sum_{\mathsf C\in\mathcal A_n(r)}
     h_{|\mathsf C|-1}\ell_{\mathsf C}
  =\widetilde L_n(r),
\]
which proves \eqref{eq:height-integral}.

Conditional on $\DTCS(n)$, the marks $E_{\mathsf C}$ of the
internal clades are independent $\Exp(1)$ variables. Thus, let
$\Gamma_m:=\sum_{j=1}^m E_j$, where $(E_j)_{j\geq1}$ are i.i.d. $\Exp(1)$ variables, independent of
$\DTCS(n)$. We therefore get, for each $r\in[n]$,
\[
  \widetilde L_n(r)\overset{\mathrm {law}}=\Gamma_{L_n(r)}
  \qquad\text{conditionally on }\DTCS(n).
\]
 Chernoff's inequality gives an absolute
constant $c>0$ such that for all $k \ge 1$ and $x>0$
\begin{equation}
    \P(\Gamma_k\leq k/2)
     \leq e^{- c k} \ \text{ and } \
    \P(|\Gamma_k-k|\geq x)
    \leq 2\exp \Bigl\{ -c\frac{x^2}{k+x} \Bigr\} . 
    \label{eq-deviation-RW}  
\end{equation} 
A union bound  yields
\begin{equation} 
    \P(\Delta_n\geq x\mid\DTCS(n))
    \leq 2\sum_{r=1}^n
         \exp \Bigl\{  
          -c\frac{x^2}{L_n(r)+x}\Bigr\}  
    \leq 2n\exp \Bigl\{ -c\frac{x^2}{L_n^*+x}\Bigr\}  .
   \label{eq-deviation-Delta}
\end{equation}

 Put
$y_n=\log n$. We claim that 
\begin{equation}\label{eq:summable-height-bound}
  \P(L_n^*> 32 y_n^2 )
  \lesssim \frac1{ n^2} .
\end{equation} 
Applying \eqref{eq-deviation-Delta} on the event 
$L_n^*\leq 32 y_n^2$    gives
\begin{equation}\label{eq:summable-path-concentration}
  \P\bigl(\Delta_n> A y_n^{3/2},\ L_n^*\leq 32 y_n^2\bigr)
  \leq2n\exp \Bigl\{  -c\frac{A^2y_n^3}{32 y_n^2+Ay_n^{3/2}} \Bigr\}.
\end{equation}
Choose $A$ large enough that $cA^2/33>3$. Then for all  large
$n$, the right-hand side of \eqref{eq:summable-path-concentration} is at
most $2n^{-2}$. Together with \eqref{eq:summable-height-bound}, this
shows that
\[
  \sum_{n\geq2}
  \P\bigl(\Delta_n>A(\log n)^{3/2}\bigr)<\infty.
\]
Applying the Borel--Cantelli Lemma proves
\eqref{eq:uniform-marked-graph-comparison}.  

It remains to show \eqref{eq:summable-height-bound}.  Using  Lemma \ref{lem:isolation-time} with $t=2y_n$ we get $\P(D_n^* > 4 y_n) \lesssim n^{-2}$. Thus \eqref{eq:summable-height-bound} follows once we prove 
\begin{equation}\label{eq:summable-height-bound-2}
  \P(L_n^*> 32 y_n^2, D_n^* \le 4y_n)
  \leq  n e^{-c y_n^2} .
\end{equation}    Indeed, since $N_n(t,U_r)\leq n$, the integral   \eqref{eq:height-integral} implies 
$
  \widetilde L_n^*\leq h_{n-1}D_n^*$. For all   large $n$, since $h_{n-1} \le 2 y_n$, 
 $  \widetilde L_n^*$ is at most $8y_n^2$ on $\{D_n^*\leq 4 y_n\}$.  
 If
$L_n^*> 32y_n^2 $ and $\widetilde L_n^*\leq8 y_n^2$, then there must exist $r \in [n]$ with $L_n(r)= L_n^*$ such that $\widetilde L_n(r)\leq8y_n^2<L_n(r)/2$. Conditioning on $\DTCS(n)$ and applying
\eqref{eq-deviation-RW}   therefore gives
\[
  \P \bigl( L_n^*> 32 y_n^2, D_n^* \le 4y_n \mid\DTCS(n)  \bigr) 
  \leq
  \sum_{ r\in[n] }
  e^{-c L_n(r)} \ind{L_n(r)>32y_n^2  }
  \leq n e^{-c  y_n^2}.
\]
This completes the proof. 
\end{proof}

\begin{proof}[Proof of Lemma \ref{lem:isolation-time}]
By Lemma~\ref{lem:finite-restriction}, the restriction to any two distinct
labels has the law of $\CTCS(2)$. Its unique internal clade has lifetime
$\Exp(h_1)=\Exp(1)$; see also
\cite[Corollary~2.6]{AldousJansonIII}. Therefore, for $r\ne k$, we have 
\begin{equation}\label{eq:pair-survival}
  \P\bigl(\I_s(U_r)=\I_s(U_k)\bigr)=e^{-s},
  \quad s\geq0.
\end{equation}
At time $s$, some label is not yet isolated if and only if some pair of
labels lies in the same fragment. By a union bound, we get $
  \P(D_n^*>s)\leq\binom n2e^{-s}$. 
Substituting $s=2\log n+t$   gives
\eqref{eq:isolation-tail}.
\end{proof}

 \subsection{Label counts versus fragment masses}
To estimate the integrand in \eqref{eq:height-integral}, we compare the binomial count
$N_n(I)$ defined in \eqref{eq:interval-occupancy} with its mean $n|I|$. We need bounds that hold simultaneously for all
intervals $I \subset (0,1)$, so that it controls all fragments in $\mathcal{I}_t$ at all time $t$.

\begin{lemma}\label{lem:all-interval-occupancy}
There are absolute constants $C_0,c_0>0$, and events
$G_n ,n\geq2$, depending only on $U_1,\ldots,U_n$, such that
$ \ind{G_n}\to1$ almost surely  and, on $G_n$, every interval
$I\subset(0,1)$ satisfies
\begin{align}
  N_n(I)&\leq C_0 ( n|I|+\log n),
  \label{eq:occupancy-upper}\\
  N_n(I)&\geq c_0n|I|
  \qquad\text{if }n|I|\geq C_0\log n.
  \label{eq:occupancy-lower}
\end{align}
\end{lemma}

\begin{proof}
We first control the counts on a dyadic grid and then pass to
arbitrary intervals. Put $y_n=\log n$  and  
$k_n=\lceil3\log_2 n\rceil$.  
For every $0\leq k\leq k_n$, let
$\mathcal D_k$ be the dyadic partition of $[0,1)$ into half-open
intervals of length $2^{-k}$. We work on the probability-one event that no $U_i$ is a
dyadic rational.

For each $D\in\mathcal D_k$, $N_n(D)$ is
binomial with mean $ n|D|=n 2^{-k}$. The   Chernoff bounds yield
an absolute constant $c>0$ such that, for $A\geq1$,
\begin{align}
   \P(N_n(D)<n|D|/2) \leq e^{-n|D|/8} \  \text{ and } \ \P(N_n(D)>2n|D|+Ay_n) \leq e^{-c(n|D|+Ay_n)}.
  \label{eq:dyadic-lower-tail}
\end{align}
 Fix $A$ so
large that $cA>5$ and $A/8>5$. Let $G_n$ denote the event 
\begin{equation}  
 \bigl\{ \   
  \forall\, D \in \cup_{k=0}^{k_n} \mathcal{D}_k \,, \    N_n(D) \leq2n|D|+Ay_n, \text{ and }
    N_n(D) \geq\tfrac12n|D|
    \text{ if } n|D|\geq Ay_n  \bigr\}.
\end{equation}
Since $\sum_{k=0}^{k_n} |\mathcal{D}_k| =\sum_{k=0}^{k_n}2^k \le 4n^3$,
by employing \eqref{eq:dyadic-lower-tail} and the union bound, we get 
$\P( G_n^c ) \lesssim n^3 e^{-5 y_n}  =n^{-2} $. The Borel--Cantelli Lemma implies $ \ind{G_n}\to1$ a.s. as required.

It remains to verify that on $G_n$, the inequalities \eqref{eq:occupancy-upper} and \eqref{eq:occupancy-lower} hold.
For the upper bound, if $|I|\geq2^{-k_n}$, choose $k\leq k_n$ with
$2^{-k-1}<|I|\leq2^{-k}$. Since $|I|\leq2^{-k}$, $I$ intersects the interiors of at most
two intervals in $\mathcal D_k$, so
\[
  N_n(I)\leq8n|I|+2Ay_n.
\]
If $|I|<2^{-k_n} \asymp n^{-3}$, use at most two intervals in $\mathcal D_{k_n}$ to cover $I$. This gives 
$N_n(I)\leq4n2^{-k_n}+2Ay_n$.
Increasing $C_0$ gives \eqref{eq:occupancy-upper}.

For the lower bound, suppose that $n|I|\geq C_0y_n$ and choose $k$ such that $|I|/4<2^{-k}\leq|I|/2$. The interval $I$
contains one  interval $D \in \mathcal{D}_k$ up to dyadic endpoints. 
Taking $C_0\geq4A$ ensures that $n|D| \ge n|I|/4\geq Ay_n$. Thus for all large $n$,  since $k \le k_n$,  the definition of $G_n$ implies 
\[
  N_n(I)\geq N_n(D)\geq\tfrac12n|D|
  \geq\tfrac18n|I| \quad \text{ on } G_n.
\]
Taking $c_0=1/8$ proves \eqref{eq:occupancy-lower}.  
\end{proof}

  Since $h_{n-1}=\log n+O(1)$,  Lemma \ref{lem:all-interval-occupancy} gives
the following bounds on the splitting rates. Write $\log_+(x):= \log (x \vee 1) $ for all $x >0$.

\begin{corollary} 
\label{cor:harmonic-logmass}
There are absolute constants $C_1,C_2<\infty$ such that, for all
 large $n$, on $G_n$ defined in Lemma \ref{lem:all-interval-occupancy}, simultaneously for all
$r\in[n]$ and $0\leq t<D_n(r)$,
\begin{align}
  h_{N_n(t,U_r)-1}
  & \leq \log_+(n|\I_t(U_r)|) +C_1\log\log n \  , \quad  \text{ and }\label{eq:harmonic-upper} \\
  h_{N_n(t,U_r)-1}&\geq\log(n|\I_t(U_r)|)-C_2 \  \text{ whenever } \quad n|\I_t(U_r)|\geq(\log n)^2.\label{eq:harmonic-lower}
\end{align}
\end{corollary}

\begin{proof}
Write $I=\I_t(U_r)$ in this proof. On $G_n$, the upper occupancy bound \eqref{eq:occupancy-upper} gives
\[
  h_{N_n(I)-1}
  \leq1+\log C_0+\log(n|I|+\log n)
  \leq\log_+(n|I|)+C_1\log\log n
\]
for all   large $n$. If $n|I|\geq(\log n)^2$,  the lower occupancy bound \eqref{eq:occupancy-lower} yields
\[
  h_{N_n(I)-1}\geq\log N_n(I)
  \geq\log(n|I|)+\log c_0.
\]
Taking $C_2=-\log c_0$ completes the proof.
\end{proof}

\subsection{The largest fragment}
\label{sec:speed}

Let $\I_t^{(1)}$ be a largest interval component of $\I_t$, choosing
the leftmost in case of ties, and set $M(t):=\log (1/|\I_t^{(1)}|)$.  Bertoin~\cite[equation~(9)]{BertoinAsymptotic} established
the first-order growth of $M(t)$, while
Kyprianou, Lane, and M{\"o}rters~\cite[Theorem~2.2]{KyprianouLaneMorters}
identified its logarithmic correction.
In the following Lemma we need the   series representation of the digamma function 
\[  \psi(\theta)  =  \frac{\dif}{\dif \theta} \log ( \Gamma(\theta)) =- \gamma + \sum_{k=0}^{\infty} \Bigl(   \frac1{k+1}-\frac1{k+\theta}\Bigr) \ ,  \quad \theta > 1. \]
This follows from the Weierstrass product for the gamma function.

\begin{lemma} \label{lem:largest-fragment-speed}
 The equation $
  \theta\psi'(\theta)=\kappa(\theta)$ 
has a unique solution $\theta_*$ on $(1,\infty)$.  Moreover,  
\begin{equation}\label{eq:largest-fragment-speed}
 \frac{M(t)}t \,\xrightarrow[t \to \infty]{\mathrm{a.s.},\, L^1 } \     v_*:=\kappa'(\theta_*)
      =\max_{\theta>1}\frac{\kappa(\theta)}{\theta} >0.
\end{equation} 
\end{lemma}

\begin{proof} 
By \cite[equations~(2) and~(4)]{BertoinAsymptotic}, the negative logarithm of the tagged-fragment mass
$X(t):=-\log|\I_t(U_1)|$ is a subordinator with Laplace exponent
\[
  \begin{aligned}
     \Phi(p)&= \int_0^1\frac{1-x^p}{1-x}\,\dif x
    =\sum_{k=0}^{\infty}\int_0^1(x^k-x^{k+p})\,\dif x\\
    &=\sum_{k=0}^{\infty} \Bigl(   \frac1{k+1}-\frac1{k+p+1}\Bigr) 
     =\psi(p+1)+\gamma=\kappa(p+1),
  \end{aligned}
\] 
Above, we expanded $(1-x)^{-1}$ as a geometric series
and used Fubini's theorem. Then, \cite[Lemma 1]{BertoinAsymptotic} gives the existence and uniqueness of $\theta_*$, and also yields the identity $\kappa'(\theta_*)
      =\max_{\theta>1}\frac{\kappa(\theta)}{\theta} .$ Moreover,  \cite[equation~(9)]{BertoinAsymptotic} shows the almost-sure
convergence in \eqref{eq:largest-fragment-speed}, since 
our fragmentation process is homogeneous with zero erosion, and its nonzero
conservative dislocation measure satisfies
$\int_0^1\min\{x,1-x\}\,\rho(\dif x)=\log2<\infty$. 

To prove the $L^1$ convergence, note that $0\leq M(t)\leq X(t)$.
Differentiating $\E[e^{-pX(t)}]=e^{-t\kappa(p+1)}$ twice at $p=0+$
and using $\kappa(1)=0$ gives, for $t>0$,
\[
  \E \bigl[  (M(t)/t)^2 \bigr]
  \leq\E \bigl[  (X(t)/t)^2 \bigr]
  =\kappa'(1)^2-\frac{\kappa''(1)}t<\infty.
\]
This completes the proof.
\end{proof}

\section{Proof of the main theorem}
\label{sec:proof-main}

\subsection{The upper bound}
\label{sec:upper-bound} 
The idea is simply to use the size of the largest fragment at each time to bound the splitting rates uniformly over all sampled lineages.

\begin{proposition}\label{prop:Q-upper}
Almost surely,
\begin{equation}\label{eq:Q-upper}
  \limsup_{n\to\infty}\frac{\widetilde L_n^*}{(\log n)^2}
  \leq\frac1{2v_*}.
\end{equation}
\end{proposition}

\begin{proof}
Put $y_n=\log n$. The tail bound \eqref{eq:isolation-tail} gives
$\P(D_n^*>4y_n)\leq(2n^2)^{-1}$. Hence the Borel--Cantelli lemma
and Lemma~\ref{lem:all-interval-occupancy} imply that
$G_n\cap\{D_n^*\leq4y_n\}$ occurs for all sufficiently large $n$
almost surely. Since $|\I_t(U_r)|\leq|\I_t^{(1)}|=e^{-M(t)}$,
the integral representation \eqref{eq:height-integral} and the
bound \eqref{eq:harmonic-upper} therefore give, almost surely for
all sufficiently large $n$,
\[
  \widetilde L_n^*
  \leq\int_0^\infty(y_n-M(t))_+\,\dif t
       +4C_1y_n\log y_n.
\]

Fix $\delta\in(0,v_*)$. By Lemma~\ref{lem:largest-fragment-speed},
there is an almost surely finite (random) time $T_\delta$ such that
$M(t)\geq(v_*-\delta)t$ for all $t\geq T_\delta$. Since $M(t)\geq0$, we have 
\[
  \int_0^\infty(y_n-M(t))_+\,\dif t
  \leq T_\delta y_n
      +\int_0^\infty\bigl(y_n-(v_*-\delta)t\bigr)_+\,\dif t
  =T_\delta \, y_n+\frac{y_n^2}{2(v_*-\delta)}.
\]
Dividing by $y_n^2$ and letting $n\to\infty$ gives, almost surely,
\[
  \limsup_{n\to\infty}\frac{\widetilde L_n^*}{(\log n)^2}
  \leq\frac1{2(v_*-\delta)}.
\]
Letting $\delta\downarrow0$ proves the claim.
\end{proof}

\subsection{The lower bound}
\label{sec:lower-bound} 
We select nested large fragments on a deterministic time mesh.   Their sizes then control   the entire   integral of splitting rates along the lineage of any sample point in the final fragment.

\begin{proposition} 
\label{prop:Q-lower}
Almost surely,
\begin{equation}\label{eq:Q-lower}
  \liminf_{n\to\infty}\frac{\widetilde L_n^*}{(\log n)^2}
  \geq\frac1{2v_*}.
\end{equation}
\end{proposition}

\begin{proof}
For a fixed block duration $T>0$, we will obtain the lower bound
$T/(2\mu_T)$ almost surely, where $\mu_T=\E M(T)$. The mean-speed limit in
Lemma~\ref{lem:largest-fragment-speed} will then allow $T$ to increase.

\smallskip\noindent
\underline{\textit{Step 1.}} We select large interval fragments greedily. 
Set $J_0=(0,1)$. Having chosen $J_{j-1}$ at time $(j-1)T$, let $J_j$
be its largest descendant at time $jT$, breaking ties by choosing the
leftmost one. The positive descendant lengths have finite sum, so the
maximum is attained.  
This construction follows the bounded-lookahead idea of
McDiarmid~\cite[Algorithm~A3]{McDiarmid1990}
and its renormalized greedy formulation in
Addario-Berry and Maillard~\cite[Section~2]{AddarioBerryMaillard}. 
Set
\begin{equation}\label{eq:block-increments}
  S_j:=-\log|J_j|,\qquad j\geq0,
\end{equation}
so that $S_0=0$. 
Conditionally on the fragmentation $\mathcal{I}_{(j-1)T}$ at $(j-1)T$, the future
inside $J_{j-1}$, rescaled to $(0,1)$, is a fresh fragmentation.
Homogeneity leaves its time scale unchanged. Hence the conditional law
of $S_j-S_{j-1}=\log\frac{|J_{j-1}|}{|J_j|}$ is that of $M(T)$ and does not depend on the past, proving
that the increments $S_j-S_{j-1}$ are i.i.d. 
Since $\mu_T=\E[ M(T)]<\infty$ by
Lemma~\ref{lem:largest-fragment-speed}, the SLLN gives
\begin{equation}\label{eq:Sj-SLLN} S_j/j\xrightarrow[j \to \infty]{\mathrm{a.s.}} \mu_T .
\end{equation}

\smallskip\noindent
\underline{\textit{Step 2.}}  We now choose one sampled leaf in the selected lineage.
Write $y_n=\log n$ and, for all  large $n$, set the terminal step 
\begin{equation}\label{eq:ay-Ny-definition}
  \mathcal{K}_n:=\max\{j\geq0:|J_j|\geq y_n^2/n\}.
\end{equation}
Since $y_n-2\log y_n\sim y_n$, the SLLN
\eqref{eq:Sj-SLLN} implies that $\mathcal{K}_n$ is finite and
$\mathcal{K}_n\sim y_n/\mu_T$ almost surely as $n\to\infty$. 
By definition, $n|J_{\mathcal{K}_n}|\geq y_n^2$, which exceeds
the threshold $C_0y_n$ in \eqref{eq:occupancy-lower} for large $n$.
Hence, on $G_n$, the terminal fragment interval contains at least
$c_0y_n^2$ labels in $[n]$:
\begin{equation}\label{eq:terminal-fragment-occupancy}
  N_n(J_{\mathcal{K}_n})
  \geq c_0n|J_{\mathcal{K}_n}| 
  \geq c_0y_n^2.
\end{equation}
The bound applies although $J_{\mathcal{K}_n}$ is random, since it holds for
every interval on the same event. Choose $r_n$ to be the smallest
label in $[n]$ satisfying $U_{r_n}\in J_{\mathcal{K}_n}$. In particular,
\begin{equation}
   \I_{jT}(U_{r_n})= J_{j}  \ , \ \text{ for every }  0 \le j \le \mathcal{K}_n.  \label{eq:J-j-I-jT}
\end{equation}
 
We have $D_n(r_n)>\mathcal{K}_nT$, because for every $0\leq t\leq\mathcal{K}_nT$, the fragment $\I_t(U_{r_n})$
contains $J_{\mathcal{K}_n}$, so
$N_n(t,U_{r_n})\geq N_n(J_{\mathcal{K}_n})\geq c_0y_n^2\geq2$.  

Moreover, for each  $1\leq j\leq \mathcal{K}_n$ and every 
$(j-1)T<t\leq jT$, the relation  \eqref{eq:J-j-I-jT} and the monotonicity of $t \mapsto \I_t$ imply
\begin{equation}\label{eq:block-lineage-envelope}
 J_j \subset   \I_t(U_{r_n}) \ , \  \text{ and } \
  \log(1/|\I_t(U_{r_n})|)\leq S_j\leq y_n-2\log y_n.
\end{equation}
\Needspace{6\baselineskip}
\smallskip\noindent
\underline{\textit{Step 3. }}
By \eqref{eq:block-lineage-envelope}, we have 
$ n|\I_t(U_{r_n})| \geq  n|J_{\mathcal{K}_n}| \ge y_n^2$ throughout $t \in [0,\mathcal{K}_nT]$.
We may therefore apply \eqref{eq:harmonic-lower} on this whole time window. Plugging it into 
\eqref{eq:height-integral} and applying \eqref{eq:block-lineage-envelope} again gives  
\begin{align}
  \widetilde L_n^*
  \geq \widetilde L_n(r_n) &\geq\sum_{j=1}^{\mathcal{K}_n}\int_{(j-1)T}^{jT}
    \bigl(\log(n|\I_t(U_{r_n})|)-C_2\bigr)\,\dif t\\
  &\geq T\sum_{j=1}^{\mathcal{K}_n}(y_n-S_j-C_2) \quad \text{ on } G_n .
  \label{eq:Q-nested-lineage-lower}
\end{align}
By \eqref{eq:Sj-SLLN}, we have $
\sum_{j=1}^{\mathcal{K}_n}S_j
   \to [{\mu_T}/{2} +o_{\mathrm{a.s.}}(1)]  {\mathcal{K}_n^{2}}$. 
Using $\mathcal{K}_n/y_n\to1/\mu_T$ a.s.,  
we obtain,  
\[  \frac1{y_n^2}\sum_{j=1}^{\mathcal{K}_n}(y_n-S_j-C_2)
   = \bigl(  
  1-\frac{C_2}{y_n}\bigr)\frac{\mathcal{K}_n}{y_n}
  - \bigl(  \frac{\mu_T}{2}+o_{\mathrm{a.s.}}(1)\bigr)
    \bigl(  \frac{\mathcal{K}_n}{y_n}\bigr)^2
   \xrightarrow[n \to \infty]{\mathrm{a.s.}}  
   \frac1{2\mu_T}. \]
Combining this limit with \eqref{eq:Q-nested-lineage-lower} and
the fact that $G_n$ occurs for all sufficiently large $n$ almost surely
by Lemma~\ref{lem:all-interval-occupancy}, we obtain, for each fixed $T>0$,
\begin{equation}\label{eq:fixed-T-lower}
  \liminf_{n\to\infty}\frac{\widetilde L_n^*}{(\log n)^2}
  \geq\frac{T}{2\mu_T}
  \quad\text{almost surely}.
\end{equation}
Taking a countable intersection, \eqref{eq:fixed-T-lower} holds on
one probability-one event simultaneously for all positive integers $T$.
Letting $T\to\infty$ along the integers and using
$\mu_T/T\to v_*$ from Lemma~\ref{lem:largest-fragment-speed} proves the desired inequality \eqref{eq:Q-lower}.
\end{proof}

\subsection{Completion of the proof}

\begin{proof}[Proof of Theorem~\ref{thm:main}]
Propositions~\ref{prop:Q-upper} and~\ref{prop:Q-lower} give $ {\widetilde L_n^*}/{(\log n)^2}\to 1/{2v_*}$ a.s.. Hence 
by Lemma~\ref{lem:marked-graph-comparison}, 
\[
  \frac{L_n^*}{(\log n)^2}
  =\frac{\widetilde L_n^*}{(\log n)^2}
   +\frac{L_n^*-\widetilde L_n^*}{(\log n)^2}
  \xrightarrow[n\to\infty]{\mathrm{a.s.}}\frac1{2v_*}=C_{\mathrm{ht}}. 
\]
 
For convergence in $L^p$, note that
$\E[\widetilde L_n(r)\mid\DTCS(n)]=L_n(r)$, so
\[
  L_n^*\leq\E[\widetilde L_n^*\mid\DTCS(n)].
\]
For every fixed $q\geq1$, the tail bound \eqref{eq:isolation-tail}
gives $\E[(D_n^*)^q]=O_q((\log n)^q)$. Since
$\widetilde L_n^*\leq h_{n-1}D_n^*$ by \eqref{eq:height-integral}, using  
conditional Jensen's inequality we get 
\[
  \E[(L_n^*)^q]
  \leq\E[(\widetilde L_n^*)^q]
  \leq h_{n-1}^q\E[(D_n^*)^q]
  =O_q((\log n)^{2q}).
\]
Taking $q>p$,  since $L_n^*/(\log n)^2$ have uniformly bounded
$q$th moments, the almost-sure convergence thus implies
convergence in $L^p$. 
The  characterization of $\theta_*$ were proved in
Lemma~\ref{lem:largest-fragment-speed}.
\end{proof}

\section*{Acknowledgement \& Statement on AI use}
H.M. thanks Oren Louidor for his  interest in this work and for several enlightening discussions.
H.M. is supported in part by a Lady Davis Fellowship at the Technion.  
  
\noindent\textbf{Statement on AI use.}  The idea of introducing the integrated splitting rate emerged from
a conversation with OpenAI's GPT-5.6 Sol about methods for studying
the largest fragments and shattering times of self-similar fragmentation processes.
The model identified the work of Dyszewski, Johnston, Palau, and
Prochno~\cite{DyszewskiJohnstonPalauProchno} and explained their use
of a mass-dependent observation times to construct a CMJ
branching process.

OpenAI's GPT-5.6 Sol and, in the later stages of the work,
GPT-6 Astra were used 
to further explore, test and refine,   the proof strategies, work out missing
computations and technical arguments, and identify potential gaps.
Anthropic's Claude Opus 5 was used for language editing and to
generate TikZ code for the figures. 
All model-assisted material included in the manuscript was critically
reviewed by the author, who independently verified all mathematical
claims and citations and takes full responsibility for the manuscript.

\bibliographystyle{amsplain}
\bibliography{references}

@article{AddarioBerryMaillard,
  author  = {Addario-Berry, Louigi and Maillard, Pascal},
  title   = {The algorithmic hardness threshold for continuous random energy models},
  journal = {Math. Stat. Learn.},
  volume  = {2},
  number  = {1},
  pages   = {77--101},
  year    = {2019},
  doi     = {10.4171/MSL/12},
  url     = {https://doi.org/10.4171/MSL/12}
}

@incollection{AldousCladograms,
  author    = {Aldous, David},
  title     = {Probability distributions on cladograms},
  booktitle = {Random Discrete Structures},
  editor    = {Aldous, David and Pemantle, Robin},
  series    = {The IMA Volumes in Mathematics and its Applications},
  volume    = {76},
  pages     = {1--18},
  publisher = {Springer},
  address   = {New York},
  year      = {1996},
  doi       = {10.1007/978-1-4612-0719-1_1},
  url       = {https://doi.org/10.1007/978-1-4612-0719-1_1}
}

@misc{AldousJansonII,
  author        = {Aldous, David J. and Janson, Svante},
  title         = {The critical beta-splitting random tree {II}: Overview and open problems},
  howpublished  = {arXiv:2303.02529v3},
  year          = {2025},
  eprint        = {2303.02529},
  archiveprefix = {arXiv},
  primaryclass  = {math.PR},
  doi           = {10.48550/arXiv.2303.02529},
  url           = {https://arxiv.org/abs/2303.02529}
}

@misc{AldousJansonIII,
  author        = {Aldous, David J. and Janson, Svante},
  title         = {The critical beta-splitting random tree {III}: The exchangeable partition representation and the fringe tree},
  howpublished  = {arXiv:2412.09655},
  year          = {2024},
  eprint        = {2412.09655},
  archiveprefix = {arXiv},
  primaryclass  = {math.PR},
  doi           = {10.48550/arXiv.2412.09655},
  url           = {https://arxiv.org/abs/2412.09655}
}

@article{AldousJansonIV,
  author  = {Aldous, David and Janson, Svante},
  title   = {The critical beta-splitting random tree {IV}: {Mellin} analysis of leaf height},
  journal = {Electron. J. Probab.},
  volume  = {30},
  pages   = {Paper No. 69, 1--39},
  year    = {2025},
  doi     = {10.1214/25-EJP1332},
  url     = {https://doi.org/10.1214/25-EJP1332}
}

@article{AldousPittel,
  author  = {Aldous, David J. and Pittel, Boris},
  title   = {The critical beta-splitting random tree {I}: Heights and related results},
  journal = {Ann. Appl. Probab.},
  volume  = {35},
  number  = {1},
  pages   = {158--195},
  year    = {2025},
  doi     = {10.1214/24-AAP2112},
  url     = {https://doi.org/10.1214/24-AAP2112}
}

@article{Basdevant,
  author  = {Basdevant, Anne-Laure},
  title   = {Fragmentation of ordered partitions and intervals},
  journal = {Electron. J. Probab.},
  volume  = {11},
  number  = {16},
  pages   = {394--417},
  year    = {2006},
  doi     = {10.1214/EJP.v11-323},
  url     = {https://doi.org/10.1214/EJP.v11-323}
}

@article{BertoinSelfSimilar,
  author  = {Bertoin, Jean},
  title   = {Self-similar fragmentations},
  journal = {Ann. Inst. H. Poincar\'e Probab. Statist.},
  volume  = {38},
  number  = {3},
  pages   = {319--340},
  year    = {2002},
  doi     = {10.1016/S0246-0203(00)01073-6},
  url     = {https://numdam.org/item/AIHPB_2002__38_3_319_0/}
}

@article{BertoinAsymptotic,
  author  = {Bertoin, Jean},
  title   = {The asymptotic behavior of fragmentation processes},
  journal = {J. Eur. Math. Soc. (JEMS)},
  volume  = {5},
  number  = {4},
  pages   = {395--416},
  year    = {2003},
  doi     = {10.1007/s10097-003-0055-3},
  url     = {https://doi.org/10.1007/s10097-003-0055-3}
}

@misc{ChorbadzhiyskaMinchevSavov,
  author        = {Chorbadzhiyska, Yoana R. and Minchev, Martin and Savov, Mladen},
  title         = {Asymptotics for {Beta}-Splitting Trees via Homogeneous Fragmentations and Meromorphic Potential Theory},
  howpublished  = {arXiv:2608.18320},
  year          = {2026},
  eprint        = {2608.18320},
  archiveprefix = {arXiv},
  primaryclass  = {math.PR},
  doi           = {10.48550/arXiv.2608.18320},
  url           = {https://arxiv.org/abs/2608.18320}
}

@misc{DyszewskiJohnstonPalauProchno,
  author        = {Dyszewski, Piotr and Johnston, Samuel G. G. and Palau, Sandra and Prochno, Joscha},
  title         = {The largest fragment in self-similar fragmentation processes of positive index},
  howpublished  = {arXiv:2409.11795v3},
  year          = {2026},
  eprint        = {2409.11795},
  archiveprefix = {arXiv},
  primaryclass  = {math.PR},
  doi           = {10.48550/arXiv.2409.11795},
  url           = {https://arxiv.org/abs/2409.11795}
}

@article{FangZeitouni,
  author  = {Fang, Ming and Zeitouni, Ofer},
  title   = {Consistent minimal displacement of branching random walks},
  journal = {Electron. Commun. Probab.},
  volume  = {15},
  pages   = {106--118},
  year    = {2010},
  doi     = {10.1214/ECP.v15-1533},
  url     = {https://doi.org/10.1214/ECP.v15-1533}
}

@article{FaraudHuShi,
  author  = {Faraud, Gabriel and Hu, Yueyun and Shi, Zhan},
  title   = {Almost sure convergence for stochastically biased random walks on trees},
  journal = {Probab. Theory Related Fields},
  volume  = {154},
  number  = {3--4},
  pages   = {621--660},
  year    = {2012},
  doi     = {10.1007/s00440-011-0379-y},
  url     = {https://doi.org/10.1007/s00440-011-0379-y}
}

@article{HaasMiermontPitmanWinkel,
  author  = {Haas, B{\'e}n{\'e}dicte and Miermont, Gr{\'e}gory and Pitman, Jim and Winkel, Matthias},
  title   = {Continuum tree asymptotics of discrete fragmentations and applications to phylogenetic models},
  journal = {Ann. Probab.},
  volume  = {36},
  number  = {5},
  pages   = {1790--1837},
  year    = {2008},
  doi     = {10.1214/07-AOP377},
  url     = {https://doi.org/10.1214/07-AOP377}
}

@article{IksanovHarmonicDescent,
  author  = {Iksanov, Alexander},
  title   = {The harmonic descent chain and regenerative composition structures},
  journal = {Electron. Commun. Probab.},
  volume  = {30},
  pages   = {Paper No. 11, 1--3},
  year    = {2025},
  doi     = {10.1214/25-ECP661},
  url     = {https://doi.org/10.1214/25-ECP661}
}

@article{IksanovNikitinYakymiv,
  author  = {Iksanov, Alexander and Nikitin, Anatolii and Yakymiv, Roman},
  title   = {An infinite balls-in-boxes approach to the critical beta-splitting tree: Some central limit theorems},
  journal = {Electron. Commun. Probab.},
  volume  = {31},
  pages   = {Paper No. 15, 1--10},
  year    = {2026},
  doi     = {10.1214/26-ECP766},
  url     = {https://doi.org/10.1214/26-ECP766}
}

@article{KolesnikContraction,
  author  = {Kolesnik, Brett},
  title   = {Critical beta-splitting, via contraction},
  journal = {Electron. Commun. Probab.},
  volume  = {30},
  pages   = {Paper No. 10, 1--14},
  year    = {2025},
  doi     = {10.1214/25-ECP658},
  url     = {https://doi.org/10.1214/25-ECP658}
}

@article{Joseph,
  author  = {Joseph, Adrien},
  title   = {A phase transition for the heights of a fragmentation tree},
  journal = {Random Structures Algorithms},
  volume  = {39},
  number  = {2},
  pages   = {247--274},
  year    = {2011},
  doi     = {10.1002/rsa.20340},
  url     = {https://doi.org/10.1002/rsa.20340}
}

@article{KyprianouLaneMorters,
  author  = {Kyprianou, Andreas E. and Lane, Francis and M{\"o}rters, Peter},
  title   = {The largest fragment of a homogeneous fragmentation process},
  journal = {J. Stat. Phys.},
  volume  = {166},
  number  = {5},
  pages   = {1226--1246},
  year    = {2017},
  doi     = {10.1007/s10955-017-1714-1},
  url     = {https://doi.org/10.1007/s10955-017-1714-1}
}

@incollection{McDiarmid1990,
  author    = {McDiarmid, C. J. H.},
  title     = {Probabilistic analysis of tree search},
  booktitle = {Disorder in Physical Systems},
  editor    = {Grimmett, G. R. and Welsh, D. J. A.},
  publisher = {Oxford University Press},
  address   = {Oxford},
  pages     = {249--260},
  year      = {1990},
  url       = {https://www.statslab.cam.ac.uk/~grg1000/books/hammfest/15-cjhm.pdf}
}

@article{Pittel1994,
  author  = {Pittel, Boris},
  title   = {Note on the heights of random recursive trees and random {$m$}-ary search trees},
  journal = {Random Structures Algorithms},
  volume  = {5},
  number  = {2},
  pages   = {337--347},
  year    = {1994},
  doi     = {10.1002/rsa.3240050207},
  url     = {https://doi.org/10.1002/rsa.3240050207}
}

@book{ShiBRW,
  author    = {Shi, Zhan},
  title     = {Branching Random Walks},
  series    = {Lecture Notes in Mathematics},
  volume    = {2151},
  publisher = {Springer},
  address   = {Cham},
  year      = {2015},
  note      = {\'Ecole d'\'Et\'e de Probabilit\'es de Saint-Flour XLII--2012},
  doi       = {10.1007/978-3-319-25372-5},
  url       = {https://igor-kortchemski.perso.math.cnrs.fr/MAP575/docs/brw.pdf}
}

@misc{MaCTCS,
  author        = {Ma, Heng},
  title         = {Extremal separation times and clade-count dynamics in fragmentation trees: a freezing transition},
  howpublished  = {arXiv:2609.14325},
  year          = {2026},
  eprint        = {2609.14325},
  archiveprefix = {arXiv},
  primaryclass  = {math.PR},
  doi           = {10.48550/arXiv.2609.14325},
  url           = {https://arxiv.org/abs/2609.14325}
}

\end{document}